\documentclass[10pt]{amsart}

\usepackage{amsmath}
\usepackage{amsthm}
\usepackage{amssymb}
\usepackage{enumerate}
\usepackage{fullpage}
\usepackage[all]{xy}
\usepackage{hyperref}
\usepackage{url}
\usepackage{wasysym}
\usepackage{verbatim}

\usepackage{tikz}
\usetikzlibrary{patterns,snakes}
\usetikzlibrary{math}

\usepackage{graphicx}
\usepackage{rotating}
\usepackage{pdflscape}

\usepackage{float}
\floatstyle{boxed}
\restylefloat{figure}

\usepackage{color}
\usepackage{bm}

\theoremstyle{definition}
\newtheorem{thm}{Theorem}[section]
\newtheorem{distalcriterion}[thm]{Distal Criterion}
\newtheorem{theorem}[thm]{Theorem}
\newtheorem{definition}[thm]{Definition}
\newtheorem{lemma}[thm]{Lemma}

\newtheorem{cor}[thm]{Corollary}
\newtheorem{prop}[thm]{Proposition}

\newtheorem{remark}[thm]{Remark}

\newtheorem{example}[thm]{Example}
\newtheorem{fact}[thm]{Fact}

\newcommand\N{\mathbb{N}}
\newcommand\Z{\mathbb{Z}}
\newcommand\Q{\mathbb{Q}}
\newcommand\R{\mathbb{R}}
\newcommand\M{\mathbb{M}}
\renewcommand\L{\mathcal{L}}
\DeclareMathOperator\id{id}

\DeclareMathOperator\Th{Th}

\DeclareMathOperator\conv{conv}

\DeclareMathOperator\DP{dp}

\author{Allen Gehret}
\author{Elliot Kaplan}
\email{allen@math.ucla.edu}
\email{eakapla2@illinois.edu}
\title{Distality for the asymptotic couple of the field of logarithmic transseries}
\address{Department of Mathematics, University of California, Los Angeles, Los Angeles, CA 90095}
\address{Department of Mathematics, University of Illinois at Urbana-Champaign, Urbana, IL 61801}
\date{\today}
\begin{document}
\maketitle

\begin{abstract}
We show that the theory $T_{\log}$ of the asymptotic couple of the field of logarithmic transseries is distal. As distal theories are NIP (= the non-independence property), this provides a new proof that $T_{\log}$ is NIP. Finally, we show that $T_{\log}$ is not strongly dependent, and in particular, it is not $\DP$-minimal and it does not have finite $\DP$-rank.
\end{abstract}

\setcounter{tocdepth}{1}
\tableofcontents

\section{Introduction}

\noindent
Distal theories and structures were introduced by Simon~\cite{SimonDistal} as a way to distinguish those NIP theories which are in some sense \emph{purely unstable}, i.e., where absolutely no stable behavior of any kind occurs. 
We sometimes think of distality as meaning: everything in sight is completely controlled by linear orders, either overtly or covertly.
Any o-minimal theory is distal, and the $p$-adic fields are distal as well.
In an o-minimal structure, everything is controlled by the obvious underlying linear order. In the $p$-adics, there is no underlying linear order, however everything is still controlled in some sense by the totally ordered value group (up to a finite residue field).
A non-example is the theory of algebraically closed valued fields (ACVF). 
Indeed, the interpretable residue field is an algebraically closed field, a purely stable structure which is not being controlled by any linear order.

\medskip\noindent
More recently, Chernikov, Galvin and Starchenko showed that a strong Szemer\'{e}di-type
 regularity lemma and other combinatorial results hold
in all distal structures~\cite{ChernikovGalvinStarchenko,ChernikovStarchenko}.
Consequently, there has been increased interest in classifying which NIP structures are distal, as well as classifying which NIP structures have distal expansions.
In this paper we prove that a particular structure, \emph{the asymptotic couple $(\Gamma_{\log},\psi)$ of the ordered valued differential field $\mathbb{T}_{\log}$ of logarithmic transseries}, is distal.
Asymptotic couples arise as the value groups of certain types of valued differential fields: the so-called \emph{asymptotic fields}. See~\cite{ADAMTT} for the full story.
We now define the object $(\Gamma_{\log},\psi)$:

\medskip\noindent
Throughout, $m$ and $n$ range over $\N=\{0,1,2,\ldots\}$. Let $\bigoplus_n\R e_n$ be a vector space over $\R$ with basis $(e_n)$. Then $\bigoplus_n\R e_n$ can be made into an ordered group using the usual lexicographical order, i.e., by requiring for nonzero $\sum_i r_ie_i$ that
\[
\textstyle \sum r_ie_i>0\ \Longleftrightarrow\ r_n>0\ \ \text{for the least $n$ such that $r_n\neq 0$.}
\]
Let $\Gamma_{\log}$ be the above ordered abelian group $\bigoplus_n\R e_n$. It is often convenient to think of an element $\sum r_ie_i$ as the vector $(r_0,r_1,r_2,\ldots)$. We follow Rosenlicht~\cite{rosenlicht} in taking the function
\[
\psi:\Gamma_{\log}\setminus\{0\} \to\Gamma_{\log}
\]
defined by
\[
(\underbrace{0,\ldots,0}_{n},\underbrace{r_n}_{\neq 0},r_{n+1},\ldots) \mapsto (\underbrace{1,\ldots,1}_{n+1},0,0,\ldots)
\]
as a new primitive, calling the pair $(\Gamma_{\log},\psi)$ an \emph{asymptotic couple} (the asymptotic couple of $\mathbb{T}_{\log}$). In~\cite{gehretQE,GehretNIP}, the model theory of $(\Gamma_{\log},\psi)$ is studied in detail. There, $(\Gamma_{\log},\psi)$ is construed as an $\L_{\log}$-structure for a certain first-order language $\L_{\log}$. In this paper we continue the study of the theory $T_{\log} = \Th_{\L_{\log}}(\Gamma_{\log},\psi)$. The main result is the following:

\begin{theorem}
\label{tlogdistalthm}
$T_{\log}$ is distal.
\end{theorem}

\noindent
An immediate consequence of Theorem~\ref{tlogdistalthm} and Proposition~\ref{distalNIP} below is:

\begin{cor}
\label{tlogNIP}
$T_{\log}$ is NIP.
\end{cor}

\noindent
This provides a new proof of the main result from~\cite{GehretNIP}. The original proof that $T_{\log}$ is NIP in~\cite{GehretNIP} involved a counting-types argument which invoked a consistency result of Mitchell and used the fact that the statement ``$T_{\log}$ is NIP'' is absolute. The appeal of the new proof of Corollary~\ref{tlogNIP} is that it is algebraic, and avoids any set-theoretic black boxes by taking place entirely within ZFC.

\medskip\noindent
Theorem~\ref{tlogdistalthm}, together with~\cite[Corollary 6.3]{ChernikovStarchenko}, also has the amusing consequence:

\begin{cor}
No model of $T_{\log}$ interprets an infinite field of positive characteristic.
\end{cor}

\noindent
We believe that no model of $T_{\log}$ interprets a field of characteristic zero either, although we leave that story for another time and place.

\medskip\noindent
In Section~\ref{distalNIPsection} we recall some definitions and basic facts around distality and NIP.
We also state and prove a general criterion for showing that a theory of a certain form is distal. This criterion is based on one developed by Hieronymi and Nell~\cite{Hieronymi_Nell}. There they use it to show that certain ``pairs'' such as $(\R;0,+,\cdot,<,2^{\Z})$ are distal. We are able to adapt it for use in our setting due to certain superficial syntactic similarities between our structure and their pairs.

\medskip\noindent
In Section~\ref{ACsection} we discuss the basics of $H$-asymptotic couples with asymptotic integration. We also define the language $\L_{\log}$ and discuss the theory $T_{\log}$. We restate some useful facts about models of $T_{\log}$ which were established in~\cite{gehretQE,GehretNIP}.
In Section~\ref{AClemmas}, we go on to prove some additional lemmas concerning the behavior of indiscernible sequences in models of $T_{\log}$.

\medskip\noindent
In Section~\ref{spreadout} we introduce a concept of an indiscernible sequence $(a_i)_{i\in I}$ in a model of $T_{\log}$ being \emph{spread out} by a parameter $b$. Roughly speaking, this means that the sequence $(a_i-b)_{i\in I}$ is sufficiently widely distributed in the convex hull of the $\Psi$-set. We then proceed to show that in such a situation, there is a certain desirable monotone interaction between the translated sequence and the $\Psi$-set (Lemma~\ref{seqmonotone}). This is one of the key steps in the proof of distality of $T_{\log}$.

\medskip\noindent
In Section~\ref{extensions} we prove several finiteness results concerning finite rank extensions of the underlying groups of models of $T_{\log}$, and their relationship to the functions $\psi, s$, and $p$ from $\L_{\log}$.

\medskip\noindent
In Section~\ref{distalproof} we bring everything together to prove Theorem~\ref{tlogdistalthm}.

\medskip\noindent
Finally, in Section~\ref{dpranksection} we prove that $T_{\log}$ is not strongly dependent (Theorem~\ref{TACnotstrong}). 
This section does not rely on any of the previous sections. 
We include this result to contrast it with Theorem~\ref{tlogdistalthm}. This also illustrates that distal structures can still be quite complicated: among NIP structures, being \emph{not strongly dependent} is more complicated than the tamer notions of \emph{strongly dependent}, having \emph{finite $\DP$-rank}, and being \emph{$\DP$-minimal}.


\subsection*{Ordered set conventions} By ``ordered set'' we mean ``totally ordered set''.

\medskip\noindent 
Let $S$ be an ordered set. Below, the ordering on $S$ will be denoted by $\leq$, and a subset of $S$ is viewed as ordered by the induced ordering. We put $S_{\infty}:= S\cup\{\infty\}$, $\infty\not\in S$, with the ordering on $S$ extended to a (total) ordering on $S_{\infty}$ by $S<\infty$.
Suppose that $B$ is a subset of $S$. We put $S^{>B}:= \{s\in S: s>b \text{ for every } b\in B\}$ and we denote $S^{>\{a\}}$ as just $S^{>a}$; similarly for $\geq, <$, and $\leq$ instead of $>$.
For $A\subseteq S$ we let
\[
\operatorname{conv}(A)\ :=\ \{x\in S: a\leq x\leq b \text{ for some } a,b\in A\}
\]
be the \textbf{convex hull of $A$} in $S$, that is, the smallest convex subset of $S$ containing $A$. For $A\subseteq S$ we put
\[
A^{\downarrow} \ := \ \{s\in S: s\leq a \text{ for some } a\in A\},
\]
which is the smallest downward closed subset of $S$ containing $A$.

\medskip\noindent
We say that $S$ is a \textbf{successor set} if every element $x\in S$ has an \textbf{immediate successor} $y\in S$, that is, $x<y$ and for all $z\in S$, if $x<z$, then $y\leq z$. For example, $\N$ and $\Z$ with their usual orderings are successor sets. We say that $S$ is a \textbf{copy of $\Z$} if $(S,<)$ is isomorphic to $(\Z,<)$.

\subsection*{Ordered abelian group conventions} Suppose that $G$ is an ordered abelian group. Then we set $G^{\neq}:= G\setminus\{0\}$. Also $G^{<}:= G^{<0}$; similarly for $\geq, \leq$, and $>$ instead of $<$. We define $|g|:= \max(g,-g)$ for $g\in G$. For $a\in G$, the \textbf{archimedean class} of $a$ is defined by
\[
[a] \ :=  \ \{g\in G : |a|\leq n|g| \text{ and } |g|\leq n|a| \text{ for some } n\geq 1\}.
\]
The archimedean classes partition $G$. Each archimedean class $[a]$ with $a\neq 0$ is the disjoint union of the two convex sets $[a]\cap G^{<}$ and $[a]\cap G^{>}$. We order the set $[G]:= \{[a]:a\in G\}$ of archimedean classes by
\[
[a]<[b] \ :\Longleftrightarrow \ n|a|<|b| \text{ for all } n\geq 1.
\]
We have $[0]<[a]$ for all $a\in G^{\neq}$, and
\[
[a]\leq [b] \ \Longleftrightarrow \ |a|\leq n|b| \text{ for some } n\geq 1.
\]

\subsection*{Model theory conventions} In general we adopt the model theoretic conventions of Appendix B of~\cite{ADAMTT}. In particular, $\L$ can be a many-sorted language. For a complete $\L$-theory $T$, we will sometimes consider a model $\M\models T$ and a cardinal $\kappa(\M)>|\L|$ such that $\M$ is $\kappa(\M)$-saturated and every reduct of $\M$ is strongly $\kappa(\M)$-homogeneous. Such a model is called a \textbf{monster model} of $T$. In particular, every model of $T$ of size $\leq\kappa(\M)$ has an elementary embedding into $\M$. All variables are finite multivariables. By convention we will write ``indiscernible sequence'' when we mean ``$\emptyset$-indiscernible sequence''.

\subsection*{Sequence conventions}

Suppose that $(a_i)_{i\in I}$ is a sequence of distinct elements from some set indexed by a linear order $I$. Given a subset or subsequence $A\subseteq (a_i)$, we let $I^{>A}$ denote the index set 
\[
 I^{>A}\ :=\ \bigcap_{a_{i_0}\in A}\{i\in I:i>i_0\}\ \subseteq \ I.
\]
Similarly for $I^{<A}$. Furthermore, given $I_0\subseteq I$ we denote by $A\cap I_0$ the set
\[
A\cap I_0 \ := \ \{a_i\in A: i\in I_0\} \ \subseteq \ A.
\]

\section{Distality and NIP}
\label{distalNIPsection}

\noindent
This section contains all of the general model-theoretic content we need for this paper. This includes a definition of distality, a criterion for proving that theories of a certain form are distal, and a proof that distal theories are NIP.
\emph{Throughout this section $\L$ is a language and $T$ is a complete $\L$-theory.}

\subsection*{Definition of distality}
\emph{In this subsection we fix a monster model $\M$ of $T$. We also let $I_1,I_2$ range over infinite linearly ordered index sets.} The definitions do not depend on the choice of this monster model. We define \emph{distality} in Definition~\ref{distaldef} below in terms of ``upgradability'' of a certain indiscernible sequence configuration. In practice, this seems to be one of the more convenient definitions to work with, and it is the only one we use in this paper. For other equivalent definitions of distality see~\cite{SimonDistal} or~\cite[Chapter 9]{SimonNIP}.

\begin{definition}
\label{distaldef}
Given $I_1$ and $I_2$, we say that $T$ is \textbf{$I_1,I_2$-distal} if for every $A\subseteq \M$, for every  $x$, and for every indiscernible sequence $(a_i)_{i\in I}$  from $\M_x$, if
\begin{enumerate}
\item $I = I_1+(c)+I_2$, and
\item $(a_i)_{i\in I_1+I_2}$ is $A$-indiscernible,
\end{enumerate}
then $(a_i)_{i\in I}$ is $A$-indiscernible. 
We say $T$ is \textbf{distal} if $T$ is $I_1,I_2$-distal for every $I_1$ and $I_2$.
Finally, we say that an $\L$-structure $\bm{M}$ is \textbf{distal} if $\Th(\bm{M})$ is distal.
\end{definition}

\noindent
It is also convenient to define what it means for a formula $\varphi(x;y)$ to be distal:

\begin{definition}
\label{distaldefformula}
Given $I_1$ and $I_2$, we say a formula $\varphi(x;y)$ is \textbf{$I_1,I_2$-distal}
if for every $b\in\M_y$ and every indiscernible sequence $(a_i)_{i\in I}$ from $\M_{x}$ such that
\begin{enumerate}
\item $I = I_1+(c)+I_2$, and
\item $(a_i)_{i\in I_1+I_2}$ is $b$-indiscernible,
\end{enumerate}
then 
$
\models \varphi(a_c;b)\leftrightarrow \varphi(a_i;b)
$
for every $i\in I$.
We say that the formula $\varphi(x;y)$ is \textbf{distal} if it is $I_1,I_2$-distal for every $I_1$ and $I_2$.
\end{definition}

\noindent
It is well known that when checking distality, either for an individual formula $\varphi(x;y)$ or an entire theory, one is free to use any specific $I_1$ and $I_2$ they wish. 
To make this sentiment precise we have introduced the provisional terminology ``$I_1,I_2$-distal'' which is not standard; see Lemmas~\ref{phixydistalequivalences} and~\ref{equivdistaldefs}.
We exploit this freedom in the proof of Distal Criterion~\ref{distal_multi} below.

\begin{lemma}
\label{phixydistalequivalences}
The following are equivalent for a formula $\varphi(x;y)$:
\begin{enumerate}
\item $\varphi(x;y)$ is distal;
\item $\varphi(x;y)$ is $I_1,I_2$-distal for some $I_1$ and $I_2$;
\end{enumerate}
\end{lemma}
\begin{proof}
\emph{The Standard Lemma}~\cite[Lemma 5.1.3]{TentZiegler} allows one to convert a counterexample of $I_1,I_2$-distality into a counterexample of $J_1,J_2$-distality, where $J_1,J_2$ are two other infinite linear orders. The details are left to the reader.
\end{proof}

\begin{lemma}
\label{equivdistaldefs}
The following are equivalent:
\begin{enumerate}
\item $T$ is distal;
\item there are $I_1$ and $I_2$ such that $T$ is $I_1,I_2$-distal;
\item every $\varphi(x;y)\in\L$ is distal;
\item there are $I_1$ and $I_2$ such that every $\varphi(x;y)\in\L$ is $I_1,I_2$-distal.
\end{enumerate}
\end{lemma}
\begin{proof}
(1)$\Rightarrow$(2)$\Rightarrow$(4)$\Leftrightarrow$(3) follow by definition and Lemma~\ref{phixydistalequivalences}. For (3)$\Rightarrow$(1), assume $T$ is not distal. Then there are $I_1$ and $I_2$ such that $T$ is not $I_1,I_2$-distal. This failure of $I_1,I_2$-distality is witnessed by some $x,y$, some sequence $(a_i)_{i\in I_1+(c)+I_2}$ from $\M_x$, some formula $\varphi(x_1,\ldots,x_n;y)$ where each $x_i$ is similar to $x$, and some parameter $b\in\M_y$. By adjusting this counterexample, through a combination of joining outer elements of the sequence with the parameter $b$, and/or grouping elements of the sequence together to create a new `thickened' sequence, one arrives at a formula $\varphi(x';y')$ (same formula, possibly different presentation of free variables) which is not distal. The argument is routine and left to the reader, although a word of caution is in order: in general $I_1$ and $I_2$ are not dense linear orders, and they may or may not have endpoints, etc. So the reduction as described above really depends on $I_1$ and $I_2$ and where the elements from $(a_i)$ which witness the failure of distality are located on these sequences.
\end{proof}

\noindent
Lemma~\ref{equivdistaldefs} permits us to work with \emph{any} $I_1,I_2$ we wish. It will be convenient for us to work with $I_1,I_2$ of a special form:

\begin{definition}
We say that a linear order $I = I_1+(c)+I_2$ is in \textbf{distal configuration at $c$} if $I_1$ and $I_2$ are infinite, $I_1$ does not have a greatest element, and $I_2$ does not have a least element.
\end{definition}

\noindent
Working with sequences in distal configuration is primarily used in the proof of Distal Criterion~\ref{distal_multi} and Proposition~\ref{pnonconstantinfty} below. It is not clear how to remove the assumption of distal configuration from these arguments, at least without making things more complicated.

\subsection*{A criterion for distality}
\label{distal_crit_sec}

\noindent
To set the stage for Distal Criterion~\ref{distal_multi} below, we now consider an extension $\L(\mathfrak{F}):=\ \L\cup\mathfrak{F}$ of the language $\L$ by a set $\mathfrak{F}$ of new unary function symbols involving sorts which are already present in $\L$. We also consider $T(\mathfrak{F})$, a complete $\L(\mathfrak{F})$-theory extending $T$. Given a model $\bm{M}\models T$ we denote by $(\bm{M},\mathfrak{F})$ an expansion of $\bm{M}$ to a model of $T(\mathfrak{F})$.
For a subset $X$ of a model $\bm{M}$, we let $\langle X \rangle$ denote the $\L$-substructure of $\bm{M}$ generated by $X$. If $\bm{M}$ is a submodel of $\bm{N}$, we let $\bm{M}\langle X \rangle$ denote $\langle M \cup X \rangle\subseteq\mathbf{N}$. For this subsection we also fix a monster model $\M$ of $T(\mathfrak{F})$. Note that $\M\!\upharpoonright\!\L$ is then a monster model of $T$.

\medskip\noindent
Distal Criterion~\ref{distal_multi} is a many-sorted, many-function generalization of~\cite[Theorem 2.1]{Hieronymi_Nell}. We give a proof below. \emph{In the statement of~\ref{distal_multi} and its proof,  $x,x',x_i,y,z,w,w_i$, etc.  are variables.}

\begin{distalcriterion}[Hieronymi-Nell]
\label{distal_multi}
Suppose $T$ is a distal theory and the following conditions hold:
\begin{enumerate}
\item The theory $T(\mathfrak{F})$ has quantifier elimination.
\item For every $\mathfrak{f}\in\mathfrak{F}$, every model $(\bm{N},\mathfrak{F}) \models T(\mathfrak{F})$, every  substructure $\bm{M} \subseteq \bm{N}$ such that $\frak{g}(M) \subseteq M$ for all $\frak{g}\in\frak{F}$, every $x$, and every $c \in N_x$, there is a $y$ and $d \in \frak{f}\big(\bm{M}\langle c\rangle\big)_y$ such that
\[
\frak{f}\big(\bm{M}\langle c\rangle\big)  \subseteq \big\langle \frak{f}(M),d \big\rangle.
\]
\item For every $\frak{f}\in\frak{F}$, the following holds: suppose that $x'$ is an initial segment of $x$,  $g,h$ are $\L$-terms of arities $xy$ and $x'z$ respectively, $b_1\in \M_y$, and $b_2\in\frak{f}(\M)_z$. If $(a_i)_{i\in I}$ is an indiscernible sequence from $\frak{f}(\M)_{x'}\times \M_{x\setminus x'}$ such that
\begin{enumerate}
\item $I = I_1+(c)+I_2$ is in distal configuration at $c$, and $(a_i)_{i\in I_1+I_2}$ is $b_1b_2$-indiscernible, and
\item $\frak{f}\big(g(a_i,b_1)\big) = h(a_i,b_2)$ for every $i\in I_1+I_2$, 
\end{enumerate}
then $\frak{f}\big(g(a_c,b_1)\big) = h(a_c,b_2)$.
\end{enumerate}
Then $T(\mathfrak{F})$ is distal.
\end{distalcriterion}

We refer the reader to Figure~\ref{bookkeepingfigure} which illustrates the bookkeeping being done in the proof of Distal Criterion~\ref{distal_multi}.
\begin{proof}
Fix an infinite linear order $I=I_1+(c)+I_2$ which is in distal configuration at $c$.
By (1) and Lemma~\ref{equivdistaldefs}, it is enough to show that every quantifier-free $\L(\mathfrak{F})$-formula $\varphi(x;y)$ is $I_1,I_2$-distal.
We prove this by induction on the number  of times $e(\psi)$ that any symbol from $\mathfrak{F}$ occurs in $\psi$. 
If $e(\psi) = 0$, this follows from the assumption that $T$ is distal. 
Let $e>0$ and suppose that for all $\L(\mathfrak{L})$-formulas $\psi'$ with $e(\psi')<e$, $\psi'$ is distal. 
Let $\psi(x;y)$ be a quantifier-free $\L(\mathfrak{F})$-formula with $e(\psi) = e$.
We will show that $\psi(x;y)$ is $I_1,I_2$-distal. 
Take an indiscernible sequence $(a_i)_{i \in I}$ from $\M_{x}$ and $b \in \M_y$ such that $I = I_1+(c)+I_2$ and $(a_i)_{i \in I_1+I_2}$ is $b$-indiscernible.

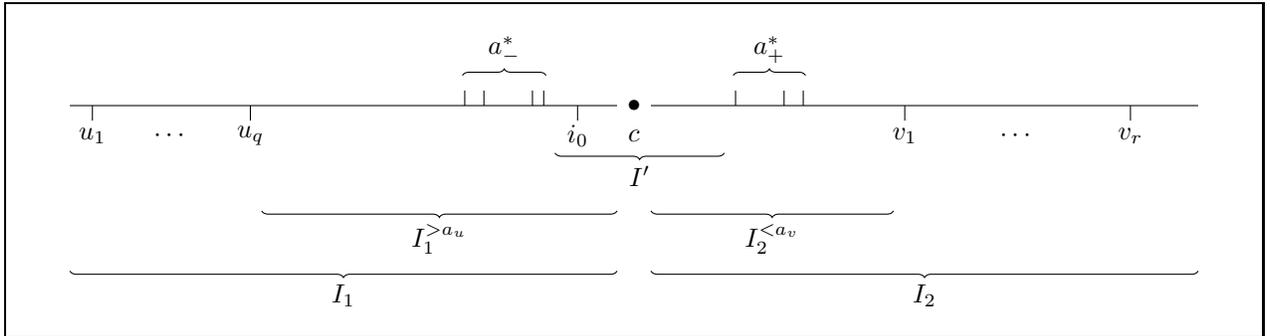
\begin{figure}[h!]
\caption{Bookkeeping in the proof of Distal Criterion~\ref{distal_multi}}
\label{bookkeepingfigure}
\begin{center}
\begin{tikzpicture}[x=1.5cm,y=1cm]
\draw (-5,0)--(-.15,0); \draw (5,0)--(.15,0);

\node at (0,0) {$\bullet$}; \node at (0,-.4){$c$};
\draw (-.5,-.2)--(-.5,0); \node at (-.5,-.4){$i_0$};

\draw (-1.5,0)--(-1.5,.2); \draw (-1.33,0)--(-1.33,.2); \draw (-.9,0)--(-.9,.2); \draw (-.8,0)--(-.8,.2);
\draw [decoration={brace, raise=0.4cm},decorate] (-1.52,0) -- (-.78,0)
node [pos=0.5,anchor=north,yshift=1.05cm] {$a^*_-$}; 

\draw (.9,0)--(.9,.2); \draw (1.33,0)--(1.33,.2); \draw (1.5,0)--(1.5,.2);
\draw [decoration={brace, raise=0.4cm},decorate] (.88,0) -- (1.52,0) 
node [pos=0.5,anchor=north,yshift=1.05cm] {$a^*_+$}; 

\draw (-3.4,0)--(-3.4,-.2); \node at (-3.4,-.4){$u_q$};
\draw (-4.8,0)--(-4.8,-.2); \node at (-4.8,-.4){$u_1$};
\node at (-4.1,-.4){$\cdots$};

\draw (2.4,0)--(2.4,-.2); \node at (2.4,-.4){$v_1$};
\draw (4.4,0)--(4.4,-.2); \node at (4.4,-.4){$v_r$};
\node at (3.4,-.4){$\cdots$};

\draw [decoration={brace, mirror, raise=0.5cm},decorate] (-.7,-.13) -- (.8,-.13) 
node [pos=0.5,anchor=north,yshift=-0.55cm] {$I'$}; 
\draw [decoration={brace, mirror, raise=0.5cm},decorate] (-3.3,-.9) -- (-.15,-.9) 
node [pos=0.5,anchor=north,yshift=-0.55cm] {$I_1^{>a_u}$}; 
\draw [decoration={brace, mirror, raise=0.5cm},decorate] (.15,-.9) -- (2.3,-.9) 
node [pos=0.5,anchor=north,yshift=-0.55cm] {$I_2^{<a_v}$}; 
\draw [decoration={brace, mirror, raise=0.5cm},decorate] (-5,-1.7) -- (-.15,-1.7) 
node [pos=0.5,anchor=north,yshift=-0.55cm] {$I_1$}; 
\draw [decoration={brace, mirror, raise=0.5cm},decorate] (.15,-1.7) -- (5,-1.7) 
node [pos=0.5,anchor=north,yshift=-0.55cm] {$I_2$}; 
\end{tikzpicture}
\end{center}
\end{figure}

Since $e>0$, there is an $\L$-term $g$ and some $\frak{f} \in \mathfrak{F}$ such that the term $\frak{f}\big(g(x;y)\big)$ appears in $\psi$. 
In other words, there is a quantifier-free $\L(\mathfrak{F})$-formula $\psi'(x;y;z)$ such that $e(\psi')<e$ and
\[
\psi(x;y)\ =\ \psi'\big(x;y;\frak{f}(g(x;y))\big).
\]
Let $\bm{M}$ be the $\L(\mathfrak{F})$-substructure of $\M$ generated by $\{a_i:i \in I_1+I_2\}$ with reduct $\bm{M}_{\L}:=\bm{M}\!\upharpoonright\!\L$. 
By (2) applied to $\bm{M}\subseteq \M\!\upharpoonright\!\L$, there is $d \in \frak{f}(\bm{M}_{\L}\langle b \rangle)_w$ for some $w$ such that
\[
\tag{A}
\frak{f}\big(\bm{M}_{\L}\langle b \rangle\big)\ \subseteq\ \big\langle \frak{f}(M),d\big\rangle.
\]

By (A), we have:
\[
\tag{B} \text{for every $i\in I_1+I_2$,}\quad\quad \mathfrak{f}\big(g(a_i;b)\big)\in \big\langle \mathfrak{f}(M),d\big\rangle
\]

Next, take $q,r\in\mathbb{N}$ and $u_{1}<\cdots<u_{q}\in I_1$ and $v_{1}<\cdots<v_{r}\in I_2$ such that
$d$ is in the $\L(\frak{F})$-structure generated by $a_ua_vb$
where $a_u := (a_{u_1},\ldots,a_{u_q})$ and $a_{v}:=(a_{v_1},\ldots,a_{v_r})$. 

Next, take $i_0\in I_1^{>a_u}$. Then applying  (B) to this $i_0$, there is an $\L$-term $h$ and $\L(\mathfrak{F})$-terms $t_1,\ldots,t_l$ (all of the form $\mathfrak{f}(s_i)$ for $\L(\frak{F})$-terms $s_i$) such that
\[
 \mathfrak{f}\big(g(a_{i_0};b)\big)\ =\ h\big(t_1(a_u,a_{i_0},a_v,a^*),\ldots,t_l(a_u,a_{i_0},a_v,a^*),d\big)
\]
where $a^*$ is a tuple of new parameters from $(a_i)_{i\in I_1+I_2}$ not yet mentioned (i.e., disjoint from $a_ua_va_{i_0}$). By $a_ua_vbd$-indiscernibility of $(a_i)_{i\in (I_1^{>a_u}+I_2^{<a_v})}$, we can arrange that $i_0$ is the largest index in $I_1$ among all indices specified so far (by sliding the elements $a^*\cap I_1^{>a_{i_0}}$ up to $I_2$). Now, let $a^*_-:= a^*\cap I_1$ and $a^*_+:= a^*\cap I_2$.

Define $I':=I_1^{>a_ua^*_-}+(c)+I_2^{<a_va^*_+}$, and note that $i_0\in I'$ and that $I'$ is also in distal configuration at $c$.
Since $(a_i)_{i\in I'\setminus(c)}$ is $a_ua_va^*bd$-indiscernible, it follows that
\[
\tag{C} \text{for every $i\in I'\setminus (c)$,}\quad\mathfrak{f}\big(g(a_i;b)\big)\ =\ h\big(t_1(a_u,a_i,a_v,a^*),\ldots,t_l(a_u,a_i,a_v,a^*),d\big).
\]

For each $i\in I'$, set
\[
a_i' \ :=\ \big( t_1(a_u,a_i,a_v,a^*),\ldots, t_l(a_u,a_i,a_v,a^*),a_i\big).
\]
$(a_i')_{i\in I'}$ is an indiscernible sequence from $\frak{f}(\M)^l\times \M_{x}$ and $(a_i')_{i\in I'\setminus (c)}$ is $a_ua_va^*bd$-indiscernible. Then by (C) and (3), it follows that
\[
\tag{D} \mathfrak{f}\big(g(a_c;b)\big)\ =\ h\big(t_1(a_u,a_c,a_v,a^*),\ldots,t_l(a_u,a_c,a_v,a^*),d\big).
\]

Finally, we note that
\begin{align*}
 &\models \psi(a_c;b) \\
\Longleftrightarrow \quad& \models \psi'\big(a_c; b ; \mathfrak{f}(g(a_c;b))\big) \quad \text{by definition of $\psi'$}\\
\Longleftrightarrow \quad& \models \psi'\big(a_c; b ;h(t_1(a_u,a_c,a_v,a^*),\ldots,t_l(a_u,a_c,a_v,a^*),d)\big) \quad\text{by (D)}\\
\Longleftrightarrow \quad& \models \psi'\big(a_i; b ;h(t_1(a_u,a_i,a_v,a^*),\ldots,t_l(a_u,a_i,a_v,a^*),d)\big)\quad (\text{for $i\in I'\setminus (c)$}) \\
&\text{by inductive hypothesis: $e(\psi') = e\big(\psi'(x;y;h(w_1,\ldots,w_l,w))\big)<e$} \\
&\text{where $w_1,\ldots,w_l$ and $w$ are variables of the appropriate sort and} \\
&\text{the partition of this last formula is $(w_1\cdots w_l x;yw)$} \\
\Longleftrightarrow\quad& \models \psi'\big(a_i; b ; \mathfrak{f}(g(a_i;b))\big)\quad (\text{for $i\in I'\setminus (c)$}) \quad\text{by (C)}\\
\Longleftrightarrow\quad&\models \psi(a_i;b)\quad (\text{for $i\in I'\setminus (c)$}).
\end{align*}
This finishes the proof.\qedhere
\end{proof}

\subsection*{Connection to NIP}
Distality was first introduced as a property which a NIP theory may or may not have~\cite{SimonDistal}.
In this paper, we have defined what it means for an arbitrary theory to be distal. This does not actually give us any additional generality since Proposition~\ref{distalNIP} below shows that every distal theory is necessarily NIP. However, this does allow us to use distality as a means for establishing that a theory is NIP, although there are more direct ways to do this (see Remark~\ref{provingNIPremark}).
\emph{Let $\M$ be a monster model of $T$.} 

\begin{definition}
We say that a partitioned $\L$-formula $\varphi(x;y)$ has the \textbf{non-independence property} (or \textbf{is NIP}) if for every $b\in\M_y$ and for every indiscernible sequence $(a_i)_{i\in I}$ from $\M_x$, there is $\varepsilon\in\{0,1\}$ and an index $i_0\in I$ such that
\[
\models \varphi(a_i;b)^{\varepsilon}\quad \text{for every $i\in I^{>i_0}$.}
\]
We say that $T$ \textbf{is NIP} if every partitioned $\L$-formula is NIP.
\end{definition}

\medskip\noindent
It is known that distality implies NIP (e.g., see~\cite[Remark 2.6]{ChernikovStarchenko}).
The proof of this fact that we include below was
communicated to us by the authors of~\cite{Hieronymi_Nell} and uses only the definitions of distality and NIP given in this paper:

\begin{prop}
\label{distalNIP}
If $T$ is distal, then $T$ is NIP.
\end{prop}
\begin{proof}
Suppose $\varphi(x;y)$ is not NIP. Then there is an indiscernible sequence $(c_n)_{n<\omega}$ in $\M_x$ and $b\in\M_y$ such that $\models \varphi(c_n;b)$ iff $n$ is even. Now define $d_n:= (c_{2n},c_{2n+1})\in\M_x\times\M_x$ and note that the sequence $(d_n)_{n<\omega}$ satisfies
\begin{enumerate}
\item the sequence $(d_{n,m})_{(n,m)\in\omega\times 2}$ with the lexicographical ordering on $\omega\times 2$ is indiscernible, and
\item for every $n$, $\models \varphi(d_{n,0};b)\wedge \neg \varphi(d_{n,1};b)$.
\end{enumerate}
By \emph{The Standard Lemma}~\cite[Lemma 5.1.3]{TentZiegler}, there is a $b$-indiscernible sequence $(e_i)_{i\in \Q}$ in $\M_x\times\M_x$ such that $\operatorname{EM}\!\big((e_i)_{i\in \Q}/b\big) = \operatorname{EM}\!\big((d_n)_{n<\omega}/b\big)$.
In particular:
\begin{enumerate}
\item the sequence $(e_{i,m})_{(i,m)\in \Q\times 2}$ with the lexicographical ordering on $\Q\times 2$ is indiscernible, and
\item for every $i$, $\models \varphi(e_{i,0};b)\wedge \neg \varphi(e_{i,1};b)$.
\end{enumerate}
Finally, for each $i\in\Q$ define the following element of $\M_x$:
\[
a_i\ :=\ \begin{cases}
e_{i,0} & \text{if $i\neq 0$} \\
e_{i,1} & \text{if $i=0$.}
\end{cases}
\]
We claim the sequence $(a_i)_{i\in\Q}$ witnesses that $\varphi(x;y)$ is not distal. Indeed:
\begin{enumerate}
\item $\Q=(-\infty,0)+(0)+(0,+\infty)$ is in distal configuration at $0$,
\item $(a_i)_{i\in (-\infty,0)+(0,+\infty)}$ is $b$-indiscernible, but
\item $\models \varphi(a_{1};b)\wedge\neg\varphi(a_0;b)$. \qedhere
\end{enumerate}
\end{proof}

\begin{remark}
\label{provingNIPremark}
Our route for proving Corollary~\ref{tlogNIP} goes through Proposition~\ref{distalNIP} and Theorem~\ref{tlogdistalthm}, which uses Distal Criterion~\ref{distal_multi}, a generalization of~\cite[Theorem 2.1]{Hieronymi_Nell}. 
If our only goal in this paper was to establish Corollary~\ref{tlogNIP}, we could have taken a slightly more direct path by generalizing in a similar way
~\cite[Theorem 4.1]{dependentpairs} to obtain a ``NIP Criterion''. In which case, we could then use ``NIP versions'' of the results in Sections~\ref{AClemmas} and~\ref{spreadout} to establish Corollary~\ref{tlogNIP}.
\end{remark}

\section{Asymptotic couples and $T_{\log}$}
\label{ACsection}

\noindent
In this section we give a summary of the basics of $H$-asymptotic couples with asymptotic integration, as well as describing the language $\L_{\log}$ and the theory $T_{\log} = \Th_{\L_{\log}}(\Gamma_{\log},\psi)$.

\subsection*{Overview of asymptotic couples}
An \textbf{asymptotic couple} is a pair $(\Gamma,\psi)$ where $\Gamma$ is an ordered abelian group and $\psi:\Gamma^{\neq}\to\Gamma$ satisfies for all $\alpha,\beta\in\Gamma^{\neq}$,
\begin{itemize}
\item[(AC1)] $\alpha+\beta\neq 0\Longrightarrow \psi(\alpha+\beta)\geq \min \big(\psi(\alpha),\psi(\beta)\big)$;
\item[(AC2)] $\psi(k\alpha) = \psi(\alpha)$ for all $k\in\Z^{\neq}$, in particular, $\psi(-\alpha) = \psi(\alpha)$;
\item[(AC3)] $\alpha>0 \Longrightarrow \alpha+\psi(\alpha)>\psi(\beta)$.
\end{itemize}
If in addition for all $\alpha,\beta\in\Gamma$,
\begin{itemize}
\item[(HC)] $0<\alpha\leq\beta\Longrightarrow \psi(\alpha)\geq\psi(\beta)$,
\end{itemize}
then $(\Gamma,\psi)$ is said to be of \textbf{$H$-type}, or to be an \textbf{$H$-asymptotic couple}.

\medskip\noindent
\emph{For the rest of this subsection, $(\Gamma,\psi)$ is an $H$-asymptotic couple.} By convention, we extend $\psi$ to all of $\Gamma$ by setting $\psi(0):=\infty$. Then $\psi(\alpha+\beta)\geq \min\big(\psi(\alpha),\psi(\beta)\big)$ holds for all $\alpha,\beta\in\Gamma$, and $\psi:\Gamma\to\Gamma_{\infty}$ is a (non-surjective) convex valuation on the ordered abelian group $\Gamma$. 
The following basic fact about valuations is used often:

\begin{fact}
If $\alpha,\beta\in\Gamma$ and $\psi(\alpha)<\psi(\beta)$, then $\psi(\alpha+\beta) = \psi(\alpha)$.
\end{fact}

\noindent
For $\alpha\in\Gamma^{\neq}$ we define $\alpha' := \alpha+\psi(\alpha)$. The following subsets of $\Gamma$ play special roles:
\[
(\Gamma^{\neq})':= \{\gamma':\gamma\in\Gamma^{\neq}\},\quad(\Gamma^{>})' := \{\gamma':\gamma\in \Gamma^{>}\},\quad (\Gamma^{<})' := \{\gamma': \gamma\in\Gamma^{<}\},\quad \Psi:=\{\psi(\gamma):\gamma\in\Gamma^{\neq}\}.
\]
We think of the map $\id+\psi:\Gamma^{\neq}\to\Gamma$ as \emph{the derivative}; this is because asymptotic couples arise in nature as the value groups of certain valued differential fields, in which case $\id+\psi$ is induced by an actual derivation. When \emph{antiderivatives} exist, they are unique:

\medskip
\begin{fact}
~\cite[Lemma 6.5.4(iii)]{ADAMTT}
The map $\gamma\mapsto \gamma'=\gamma+\psi(\gamma):\Gamma^{\neq}\to\Gamma$ is strictly increasing.
\end{fact}

\noindent
This allows us to talk about \emph{asymptotic integration}:

\begin{definition}
If $\Gamma = (\Gamma^{\neq})'$, the we say that $(\Gamma,\psi)$ has \textbf{asymptotic integration}. Suppose $(\Gamma,\psi)$ has asymptotic integration. Given $\alpha\in\Gamma$ we let $\int\alpha$ denote the unique $\beta\in\Gamma^{\neq}$ such that $\beta'=\alpha$ and we call $\beta = \int\alpha$ the \textbf{integral} of $\alpha$. This gives us a function $\int:\Gamma\to\Gamma^{\neq}$ which is the inverse of $\gamma\mapsto \gamma':\Gamma^{\neq}\to\Gamma$.
\end{definition}

\noindent
\emph{We now further assume that $(\Gamma,\psi)$ has asymptotic integration.} A closely related function to $\int$ is the \textbf{successor function} $s:\Gamma\to\Psi$ defined by $\alpha\mapsto s(\alpha):=\psi(\int\alpha)$. The successor function gets its name due to its behavior on the $\Psi$-set in the asymptotic couple $(\Gamma_{\log},\psi)$. More generally:

\begin{example}
The asymptotic couple $(\Gamma_{\log},\psi)$ is of $H$-type and has asymptotic integration. The functions $\int$ and $s$ behave as follows:
\begin{enumerate}
\item (Integral) For $\alpha = (r_0,r_1,r_2,\ldots)\in\Gamma_{\log}$, take the unique $n$ such that $r_n\neq 1$ and $r_m = 1$ for $m<n$. Then
\[
\alpha = (\underbrace{1,\ldots,1}_n,\underbrace{r_n}_{\neq 1},r_{n+1},r_{n+2},\ldots)\ \mapsto\ \textstyle\int\alpha = (\underbrace{0,\ldots,0}_n,r_n-1,r_{n+1},r_{n+2},\ldots)
\]
\item (Successor) For $\alpha = (r_0,r_1,r_2,\ldots)\in\Gamma_{\log}$, take the unique $n$ such that $r_n\neq 1$ and $r_m=1$ for $m<n$. Then
\[
\alpha = (\underbrace{1,\ldots,1}_n,\underbrace{r_n}_{\neq 1},r_{n+1},r_{n+2},\ldots)\ \mapsto\ s(\alpha) = (\underbrace{1,\ldots,1}_{n+1},0,0,\ldots)
\]
\end{enumerate}
\end{example}

\noindent
We conclude this subsection with some general facts about the $s$-function which we will need later:

\begin{fact}
\label{sfacts}
Let $\alpha,\beta\in\Gamma$. Then
\begin{enumerate}
\item\label{sinvconvex} if $\alpha\in s(\Gamma)$, then $s^{-1}(\alpha)\cap(\Gamma^{>})'$ and $s^{-1}(\alpha)\cap(\Gamma^{<})'$ are convex in $\Gamma$,
\item\label{succid} (Successor Identity) if $s\alpha<s\beta$, then $\psi(\alpha-\beta) = s\alpha$,
\item\label{poptop} if $\alpha\in (\Gamma^{<})'$ and $n\geq 1$, then $\alpha+(n+1)(s\alpha-\alpha)\in(\Gamma^{>})'$.
\end{enumerate}
\end{fact}
\begin{proof}
(\ref{sinvconvex}) is~\cite[Corollary 3.6]{gehretQE}, (\ref{succid}) is~\cite[Lemma 3.4]{gehretQE}, and (\ref{poptop}) is~\cite[Lemma 3.10]{GehretLiouville}.
\end{proof}

\subsection*{The theory $T_{\log}$}
Let $\L_{AC}$ be the ``natural'' language of asymptotic couples; $\L_{AC} = \{0,+,-,<,\psi,\infty\}$ where $0,\infty$ are constant symbols, $+$ is a binary function symbol, $-$ and $\psi$ are unary function symbols, and $<$ is a binary relation symbol. We consider an asymptotic couple $(\Gamma,\psi)$ as an $\L_{AC}$-structure with underlying set $\Gamma_{\infty}$ and the obvious interpretation of the symbols of $\L_{AC}$, with $\infty$ as a default value:
\[
-\infty \ = \ \gamma+\infty \ = \ \infty+\gamma \ = \ \infty+\infty \ = \ \psi(0) \ = \ \psi(\infty) \ = \ \infty
\]
for all $\gamma\in\Gamma$.

\medskip\noindent
Let $T_{AC}$ be the $\L_{AC}$-theory whose models are the divisible $H$-asymptotic couple with asymptotic integration such that
\begin{enumerate}
\item $\Psi$ as an ordered subset of $\Gamma$ has least element $s0$,
\item $s0>0$,
\item $\Psi$ as an ordered subset of $\Gamma$ is a successor set,
\item for each $\alpha\in\Psi$, the immediate successor of $\alpha$ in $\Psi$ is $s\alpha$, and
\item $\gamma\mapsto s\gamma:\Psi\to\Psi^{>s0}$ is a bijection.
\end{enumerate}
It is clear that $(\Gamma_{\log},\psi)$ is a model of $T_{AC}$. For a model $(\Gamma,\psi)$ of $T_{AC}$, we define the function $p:\Psi^{>s0}\to\Psi$ to be the inverse to the function $\gamma\mapsto s\gamma:\Psi\to\Psi^{>s0}$. We extend $p$ to a function $\Gamma_{\infty}\to\Gamma_{\infty}$ by setting $p(\alpha):=\infty$ for $\alpha\in\Gamma_{\infty}\setminus\Psi^{>s0}$.

\medskip\noindent
Next, let $\L_{\log} = \L_{AC}\cup\{s,p,\delta_1,\delta_2,\delta_3,\ldots\}$ where $s,$ $p$, and $\delta_n$ for $n\geq 1$ are unary function symbols. All models of $T_{AC}$ are considered as $\L_{\log}$-structures in the obvious way, again with $\infty$ as a default value, and with $\delta_n$ interpreted as division by $n$. 

\medskip\noindent
We let $T_{\log}$ be the $\L_{\log}$-theory whose models are the models of $T_{AC}$.
These are some of the main results concerning $T_{\log}$ from~\cite[\S 5]{gehretQE}:

\begin{thm}
\label{Tlogknownthms}
The $\L_{\log}$-theory $T_{\log}$
\begin{enumerate}
\item has a universal axiomatization,
\item has quantifier elimination,
\item is complete,
\item is decidable, and
\item is model complete.
\end{enumerate}
\end{thm}

\noindent
We shall also need the following facts about models of $T_{\log}$. \emph{For Facts~\ref{lemma6.8gehretQE} and~\ref{cor6.5gehretQE} we let $(\Gamma,\psi)\models T_{\log}$.}

\begin{fact}
\label{lemma6.8gehretQE}
\cite[Lemma 6.8]{gehretQE}
$\Psi$ is a linearly independent subset of $\Gamma$ as a vector space over $\Q$.
\end{fact}

\begin{fact}
\label{cor6.5gehretQE}
\cite[Corollary 6.5]{gehretQE}
Let $n\geq 1$, $\alpha_1<\cdots<\alpha_n\in\Psi$, and let $\alpha = \sum_{j=1}^n q_j\alpha_j$ for $q_1,\ldots,q_n\in\Q^{\neq}$. Then
\begin{enumerate}
\item $\sum_{j=1}^nq_j\neq 1\Longrightarrow s(\alpha) = s0$,
\item $\sum_{j=1}^nq_j=1\Longrightarrow s(\alpha) = s(\alpha_1)$.
\end{enumerate}
\end{fact}

\section{Some indiscernible lemmas}
\label{AClemmas}

\noindent
\emph{In this section $\M$ is a monster model of $T_{\log}$ with underlying set $\Gamma_{\infty}$. Furthermore, $I = I_1+(c)+I_2$ is an ordered index set with infinite $I_1$ and $I_2$, and $i,j,k$ range over $I$.} In this section we will handle most of the cases that will arise when verifying hypothesis (3) from Distal Criterion~\ref{distal_multi} in our proof of distality for $T_{\log}$.

\medskip\noindent
In terms of~\ref{distal_multi}(3), Lemma~\ref{nonconstantconstant} will handle most of the cases where the sequence $\big(g(a_i,b_1)\big)$ which gets plugged into $\frak{f}$ is nonconstant, but the output sequence $\big(h(a_i,b_2)\big)$ is constant. The arguments involved use essentially that $\psi$ is a convex valuation on $\Gamma^{\neq}$, and that $s$ behaves in a similar topological manner to $\psi$.

\begin{lemma}
\label{nonconstantconstant}
Let $(a_i)_{i\in I}$ be a nonconstant indiscernible sequence from $\Gamma$ and suppose $(b,b')\in\Gamma\times\Psi$ is such that $(a_i)_{i\in I_1+I_2}$ is $bb'$-indiscernible. Then:
\begin{enumerate}
\item if $\psi(a_i-b) = b'$ for all $i\neq c$, then $\psi(a_c-b) = b'$;
\item if $s(a_i-b) = b'$ for all $i\neq c$, then $s(a_c-b) = b'$;
\item for $\frak{f}\in\{\psi, s\}$, it cannot be the case that $\frak{f}(a_i-b) = \infty$ for infinitely many $i$;
\item it cannot be the case that $p(a_i-b) = b'$ for infinitely many $i$.
\end{enumerate}
\end{lemma}
\begin{proof}
For (1), assume $\psi(a_i-b) = b'$ for all $i\neq c$. By $bb'$-indiscernibility, the sequence $(a_i-b)_{i\in I_1+I_2}$ is contained entirely within the convex set $\Gamma^<\cap \psi^{-1}(b')$ or it is contained entirely within the convex set $\Gamma^{>}\cap\psi^{-1}(b')$. In either case, $a_c-b$ is also contained in the same convex set by monotonicity of $(a_i-b)_{i\in I}$. Thus $\psi(a_c-b) = b'$.

The argument for (2) is similar to the argument for (1), except that the two convex sets are $(\Gamma^{<})'\cap s^{-1}(b')$ and $(\Gamma^{>})'\cap s^{-1}(b)$ (see Fact~\ref{sfacts}(\ref{sinvconvex})).

(3) and (4) follow from the assumption that $(a_i)$ is nonconstant.
\end{proof}

\noindent
Lemma~\ref{nonconstantnonconstant} will handle the cases where both sequences $\big(g(a_i,b_1)\big)$ and $\big(h(a_i,b_2)\big)$ are nonconstant:

\begin{lemma}
\label{nonconstantnonconstant}
Let $(a_ia_i')_{i\in I}$ be an indiscernible sequence from $\Gamma\times\Psi$ such that $(a_i)$ and $(a_i')$ are each nonconstant, and suppose $b\in\Gamma$ is such that $(a_ia_i')_{i\in I_1+I_2}$ is $b$-indiscernible. Then:
\begin{enumerate}
\item if $\psi(a_i - b) = a_i'$ for all $i\neq c$, then $\psi(a_c-b) = a_c'$;
\item  if $s(a_i - b) = a_i'$ for all $i\neq c$, then $s(a_c-b) = a_c'$;
\item if $p(a_i-b) = a_i'$ for all $i\neq c$, then $p(a_c-b) = a_c'$.
\end{enumerate}
\end{lemma}
\begin{proof}
Without loss of generality, we will assume that $(a_i')$ is strictly increasing.

For (1), we first observe that if $i<j$, then $\psi(a_i-a_j) = a_i'$. Indeed, by indiscernibility, it suffices to show this for $i<j$ both from $I_1+I_2$. Given such $i<j$, we have
\[
\psi(a_i-a_j)\ =\ \psi\big((a_i-b)+(b-a_j)\big)\ =\ \min(a_i', a_j')\ =\ a_i'.
\]
Next, suppose $j\in I_2$, and note that
\[
\psi(a_c-b)\ =\ \psi\big((a_c-a_j)+(a_j-b)\big)\ =\ \min(a_c',a_j')\ =\ a_c'.
\]

For (2), we first claim that if $i<j$, then $\psi(a_i-a_j) = a_i'$. By indiscernibility, it suffices to show this for $i<j$ both from $I_1+I_2$. Given such $i<j$, we have $a_i' = s(a_i-b)<s(a_j-b) = a_j'$. By the Successor Identity (Fact~\ref{sfacts}(\ref{succid})), we have
\[
a_i' = s(a_i-b)\ =\ \psi\big((a_i-b) - (a_j-b)\big)\ =\ \psi(a_i-a_j).
\]
Next, suppose $j<k$ are both from $I_2$. Then by monotonicity of $(a_i-b)_{i\in I}$ and $b$-indiscernibility of $(a_i)_{i\in I_1+I_2}$,
\[
s(a_c-b)\ \leq\ s(a_j-b)\ <\ s(a_k-b).
\]
By the Successor Identity we have
\[
a_c'\ =\ \psi(a_c-a_k)\ =\ \psi\big((a_c-b)-(a_k-b)\big)\ =\  s(a_c-b).
\]

Finally for (3), we note that if $i \in I_1+I_2$, then $a_i-b \in \Psi^{>s0}$ and so
\[
a_i-b = sp(a_i-b) = s(a_i').
\]
From this we see that $a_i-s(a_i') = b$ and, in particular, $(a_i-s(a_i'))_{i \in I_1+I_2}$ is constant. By indiscernibility of $(a_ia_i')_{i \in I}$, we have $a_c - s(a_c') = b$ and so $p(a_c-b) = ps(a_c')$. As $a_i'\in \Psi$ for all $i \in I_1+I_2$ we have $a_c' \in \Psi$, and so $ps(a_c')= a_c'$. 
\end{proof}

\noindent
Lemma~\ref{constantconstant} will handle the cases where both sequences $\big(g(a_i,b_1)\big)$ and $\big(h(a_i,b_2)\big)$ are constant. This case is trivial, but it does implicitly rely on the behavior $\mathcal{L}$-terms, where $\L:=\{0,+,-,<,(\delta_n)_{n<\omega},\infty\}\subseteq\L_{\log}$:

\begin{lemma}
\label{constantconstant}
Given $\frak{f}\in\{\psi,s,p\}$, let $g,h$ be $\L$-terms of arities $n+k$ and $m+l$ respectively with $m\leq n$, $b_1\in \M^k$, $b_2\in \frak{f}(\M)^l$, $(a_i)_{i\in I}$ be an indiscernible sequence from $\frak{f}(\M)^m\times \M^{n-m}$ such that
\begin{enumerate}
\item $(a_i)_{i\in I_1+I_2}$ is $b_1b_2$-indiscernible,
\item $\frak{f}\big(g(a_i,b_1)\big) = h(a_i,b_2)$ for every $i\neq c$,
\item $\big(g(a_i,b_1)\big)_{i\in I_1+I_2}$ is a constant sequence.
\end{enumerate}
Then $\frak{f}\big(g(a_c,b_1)\big) = h(a_c,b_2)$.
\end{lemma}
\begin{proof}
First, suppose $h(a_i,b_2) = \infty$ for every $i\neq c$. Then either $\infty$ occurs in the term $h$, or in the tuple $b_2$ in a ``non-dummy'' manner, or in one of the first $m$ coordinates of $a_i$ in a ``non-dummy'' manner. In any of these cases, it follows that $h(a_c,b_2) = \infty$. 

Otherwise, suppose there is $b\in\Gamma$ such that $h(a_i,b_2) = b$ for every $i\neq c$. Then the $\L$-term $h$ essentially computes a $\Q$-linear combination of its arguments, and $\infty$ does not get involved at all. Thus $\big(h(a_i,b_2)\big)_{i\in I}$ is monotone by indiscernibility of $(a_i)_{i\in I}$. In particular, $h(a_c,b_2) = b$.

A similar argument shows that for an arbitrary $i_0\in I_1+I_2$, $g(a_c,b_1) = g(a_{i_0},b_1)$. Thus
\[
\frak{f}\big(g(a_c,b_1)\big)\ =\ \frak{f}\big(g(a_{i_0},b_1)\big)\ =\ h(a_{i_0},b_2)\ =\ h(a_c,b_2). \qedhere
\]
\end{proof}

\noindent
The following lemma greatly simplifies the sequence $\big(h(a_i,b_2)\big)$. \emph{We no longer assume that $j$ and $k$ range over $I$.}

\begin{lemma}
\label{RHSsimplification}
Let $h(x,y)$ be an $\L$-term of arity $m+n$, $b\in \M^n$ and $(a_i)_{i\in I}$ an indiscernible sequence from $\Psi_{\infty}^m$, with $a_i = (a_{i,1},\ldots,a_{i,m})$. Assume that $h(a_i,b)\in\Psi_{\infty}$ for infinitely many $i$. Then one of the following is true:
\begin{enumerate}
\item $h(a_i,b) = \infty$ for every $i$;
\item there is $\beta\in\Psi$ such that $h(a_i,b) = \beta$ for every $i$;
\item there is $l\in\{1,\ldots,m\}$ such that $h(a_i,b) = a_{i,l}$ for every $i$.
\end{enumerate}
\end{lemma}
\begin{proof}
If one of the components of $b$ which corresponds to a free variable which actually occurs in $h$ is $\infty$, then $h(a_i,b) = \infty$ for every $i$. Similarly if the constant $\infty$ occurs in the term $h$. Thus we may assume for the remainder of the proof that none of the components of $b$ are $\infty$ and that $\infty$ does not occur in the term $h$. We may then write $h(x,b) = (\textstyle\sum_{j=1}^m q_jx_j)+c$ where $c\neq \infty$ is a $\Q$-linear combination of the components of $b$ and $q_1,\ldots,q_m\in\Q$. We consider three disjoint cases:

\textbf{Case 1:} \emph{There is $i_0 \in I$ with $h(a_{i_0},b) = \infty$.} Then there must be $j \in \{1,\ldots, m\}$ with $a_{i_0,j} = \infty$, and so $a_{i,j} = \infty$ for every $i$. We conclude that $h(a_i,b) = \infty$ for every $i$.

\textbf{Case 2:} \emph{$h(a_i,b) \neq \infty$ for every $i$, and there are distinct $i_0,i_1 \in I$ and $\beta \in \Gamma$ with $h(a_{i_0},b) = h(a_{i_1},b)= \beta$.} We see then that $\sum_{j=1}^m q_ja_{i_0,j}=\sum_{j=1}^m q_ja_{i_1,j}$ and so $\sum_{j=1}^m q_ja_{i,j}=\sum_{j=1}^m q_ja_{i',j}$ for every $i, i'\in I$ by indiscernibility. We conclude that $h(a_i,b) = \beta$ for every $i$ and since $h(a_i,b)\in\Psi_{\infty}$ for infinitely many $i$, we see that $\beta \in \Psi$.

\textbf{Case 3:} \emph{$h(a_i,b) \neq \infty$ for every $i$, and $h(a_i,b) \neq h(a_{i'},b)$ for all distinct $i$ and $i'$.} 
We will first clean up the summation by removing constant and redundant sequences. By indiscernibility, there is $m_0\geq 1$ and natural numbers $1\leq \eta(1)<\cdots<\eta(m_0)\leq m$ such that
\begin{enumerate}
\item for every $j\in \{\eta(1),\ldots,\eta(m_0)\}$, the sequence $(a_{i,j})_{i\in I}$ is nonconstant,
\item for every $j,j'\in \{\eta(1),\ldots,\eta(m_0)\}$ such that $j\neq j'$, $a_{i,j}\neq a_{i,j'}$ for every $i$, and
\item given $j\in \{1,\ldots,m\}\setminus\{\eta(1),\ldots,\eta(m_0)\}$, either
\begin{enumerate}
\item the sequence $(a_{i,j})_{i\in I}$ is constant, or
\item there is $j'\in \{\eta(1),\ldots,\eta(m_0)\}$ such that $a_{i,j} = a_{i,j'}$ for every $i$.
\end{enumerate}
\end{enumerate}

By rearranging the components of $(a_i)$ and the $q_j$'s, we may assume for the rest of the proof that $\eta(j) = j$ for $j=1,\ldots,m_0$.

Next, for $j\in \{1,\ldots,m_0\}$, define
\[
\textstyle \tilde{q}_j := \sum_{j'\in A} q_{j'}, \ \text{where } A:= \{j': a_{i,j} = a_{i,j'} \text{ for every $i$}\}
\]
and
\[
\textstyle \tilde{c} := c+ \sum_{j\in B} a_{i_0,j}, \ \text{where } B:= \{j: (a_{i,j})_{i\in I} \text{ is a constant sequence}\} \ \text{and $i_0\in I$ is some fixed index.}
\]
We now have that for every $i$,
\[
\textstyle h(a_i,b) = \big(\sum_{j=1}^{m_0}\tilde{q}_ja_{i,j}\big) + \tilde{c},
\]
and for every $i,i'$ and $j,j'\in\{1,\ldots,m_0\}$, $a_{i,j}\neq a_{i',j'}$ whenever $(i,j)\neq (i',j')$.

Now choose distinct $i_0,i_1,\ldots,i_{m_0+1}\in I$ such that $h(a_{i_k},b)\in\Psi$ for $k=0,\ldots,m_0+1$. We have for each $k\in \{1,\ldots,m+1\}$ that
\[
\textstyle \big(\sum_{j=1}^{m_0}\tilde{q}_ja_{i_0,j}\big) -h(a_{i_0},b) \ = \ -\tilde{c} \ = \ \big(\sum_{j=1}^{m_0}\tilde{q}_ja_{i_k,j}\big) -h(a_{i_k},b)
\]
and so
\[
\tag{$\ast$}  \textstyle h(a_{i_k},b) \ = \ \big(\sum_{j=1}^{m_0}\tilde{q}_ja_{i_k,j}\big) - \big(\sum_{j=1}^{m_0}\tilde{q}_ja_{i_0,j}\big) + h(a_{i_0},b).
\]
Since $h(a_{i_k},b)$ and $h(a_{i_0},b)$ are distinct elements of $\Psi$, and $a_{i_0,1},\ldots,a_{i_0,m_0}, a_{i_k,1},\ldots,a_{i_k,m_0}$ are also distinct elements from $\Psi$, we deduce from the $\Q$-linear independence of $\Psi$ (Fact~\ref{lemma6.8gehretQE}) that
\[
h(a_{i_k},b) \ \in \ \{a_{i_0,1},\ldots,a_{i_0,m_0}, a_{i_k,1},\ldots,a_{i_k,m_0}\}.
\]
We claim that there is at least one $k\in\{1,\ldots,m_0+1\}$ such that $h(a_{i_k},b)\in \{a_{i_k,1},\ldots,a_{i_k,m_0}\}$. Suppose not. Then there is a function $\sigma:\{1,\ldots,m_0+1\}\to \{1,\ldots,m_0\}$ such that $h(a_{i_k},b) = a_{i_0,\sigma(k)}$. As $h(a_{i_k},b)\neq h(a_{i_{k'}},b)$ for all $k,k'\in\{1,\ldots,m_0+1\}$ such that $k\neq k'$, we must have that $\sigma$ is injective, a contradiction. Therefore we can take $k\in\{1,\ldots,m_0+1\}$ and $l\in \{1,\ldots,m_0\}$ such that $h(a_{i_k},b) = a_{i_k, l}$. In particular, $a_{i_k,l}\neq 0$ and so $a_{i_0,l}\neq 0$. Again from $(\ast)$, the $\Q$-linear independence of $\Psi$, and the fact that $a_{i_0,1},\ldots,a_{i_0,m_0}, a_{i_k,1},\ldots,a_{i_k,m_0}$ are all distinct, we deduce that $h(a_{i_0},b) = a_{i_0, l}$, that $\tilde{q}_l = 1$, and that $\tilde{q}_j = 0$ for $j\neq l$. From this, we deduce that $\tilde{c}=0$, and so $h(a_i,b) = a_{i,l}$ for every $i$.
\end{proof}

\section{Spread out sequences}
\label{spreadout}

\noindent
\emph{In this section $\M$ is a monster model of $T_{\log}$ with underlying set $\Gamma_{\infty}$. Furthermore, $I$ is an infinite ordered index set, $i,j$ range over $I$, and $(a_i)_{i\in I}$ is a strictly increasing indiscernible sequence from $\Gamma$.}

\begin{definition}
Given $a,b\in \conv(\Psi)$, we write $a\ll b$ if $s^na<b$ for every $n$.
If there is $b\in\Gamma$ such that $(a_i-b)_{i\in I}\subseteq \conv(\Psi)$, and for every $i<j$, $s0\ll a_i-b\ll a_j-b$, then we say that $(a_i)_{i\in I}$ is \textbf{spread out by $b$}.
\end{definition}

\noindent
Intuitively, the idea behind $(a_i)_{i\in I}$ being spread out by $b$, is that upon translating $(a_i)_{i\in I}$ by $b$, all elements of the sequence $(a_i-b)_{i\in I}$ are now in the convex hull of the $\Psi$-set, and each element of the sequence ``lives on'' its own copy of $\Z$ (or really, the convex hull of a copy of $\Z$). In Figure~\ref{spreadoutfigure} we suppose $(a_i)_{i\in I} = (a_n)_{n<\omega}$ is an indiscernible sequence spread out by $b$, and we illustrate the positions of $(a_n-b)_{n<5}$. Furthermore, each element  $a_i-b$ will have a ``nearest'' element of the $\Psi$-set, namely $ps(a_i-b)$. Then next lemma shows that these nearest elements do not depend on $b$.

\begin{figure}[h!]
\caption{Five elements of a sequence being spread out by $b$.}
\label{spreadoutfigure}
\begin{center}
\begin{tikzpicture}
\draw (0,0)--(15,0); 
\tikzmath{\x = .7; \y = .75; \z =2.85;
\w =1/(1-\y);}
\draw (.5,-.4)--(.5,.4);
\node at (.5,-.6) {\small$0$}; 
\node at (.5+2*\y*\z+\w*\y*\y*\z*2-2*\z*\w*\y*\y^1-\z*\y^1*\x^0, -.6) {\small$s0$}; 

\foreach \a in {0,1,2,3,4,5,6,7,8,9,10,11,12,13,14,15,16,17,18,19,20} {
\begin{scope}
\foreach \b in {1,2,3,4,5,6} {
\begin{scope}
\draw (.5+2*\y*\z+\w*\y*\y*\z*2-2*\z*\w*\y*\y^\b-\z*\y^\b*\x^\a, -.4*\y^\b*\x^\a)--(.5+2*\y*\z+\w*\y*\y*\z*2-2*\z*\w*\y*\y^\b-\z*\y^\b*\x^\a, \z*\y^\b*\x^\a);
\end{scope}}
\end{scope}}

\foreach \a in {1,2,3,4,5,6,7,8,9,10,11,12,13,14,15,16,17,18,19,20} {
\begin{scope}
\foreach \b in {2,3,4,5,6} {
\begin{scope}
\draw (.5+2*\y*\z+\w*\y*\y*\z*2-2*\z*\w*\y^\b+\z*\y^\b*\x^\a, -.4*\y^\b*\x^\a)--(.5+2*\y*\z+\w*\y*\y*\z*2-2*\z*\w*\y^\b+\z*\y^\b*\x^\a, \z*\y^\b*\x^\a);
\end{scope}}
\end{scope}}

\node (A0) at (.5+2*\y*\z+\w*\y*\y*\z*2-2*\z*\w*\y*\y^2-\z*\y^2*\x^2-\y^2*\x^2*.4, 0) {\tiny$\bullet$}; \node at ([yshift=-.3cm]A0) {\small$a_0-b$};

\node (A1) at (.5+2*\y*\z+\w*\y*\y*\z*2-2*\z*\w*\y*\y^3-\z*\y^3*\x^1-\y^3*\x^1*.2, 0) {\tiny$\bullet$}; \node at ([yshift=-.3cm]A1) {\small$a_1-b$};

\node (A2) at (.5+2*\y*\z+\w*\y*\y*\z*2-2*\z*\w*\y^4+\z*\y^4*\x^1, 0) {\tiny$\bullet$}; \node at ([yshift=-.3cm]A2) {\small$a_2-b$};

\node (A3) at (.5+2*\y*\z+\w*\y*\y*\z*2-2*\z*\w*\y^5+\z*\y^5*\x^1+\y^5*\x^1*.2, 0) {\tiny$\bullet$}; \node at ([yshift=-.3cm]A3) {\small$a_3-b$};

\node (A4) at (.5+2*\y*\z+\w*\y*\y*\z*2-2*\z*\w*\y*\y^6-\z*\y^6*\x^0+\y^6*\x^0*.4, 0) {\tiny$\bullet$}; \node at ([yshift=-.3cm]A4) {\small$a_4-b$};

\end{tikzpicture}
\end{center}

\end{figure}
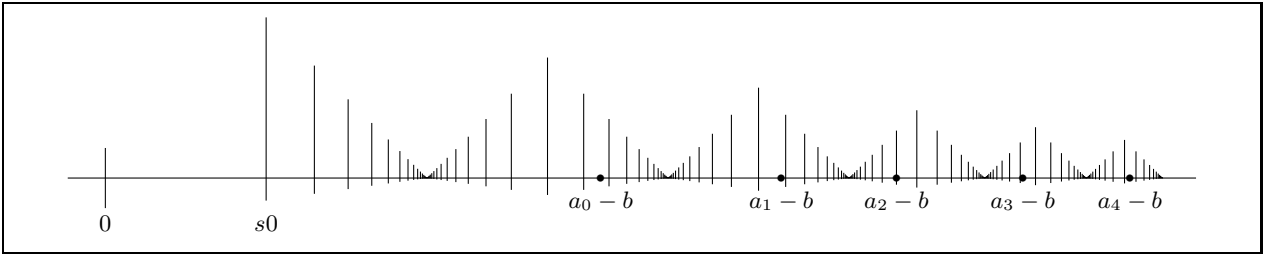

\begin{lemma}
Suppose $b_0,b_1\in\Gamma$ are such that $(a_i)_{i\in I}$ is spread out by both $b_0$ and $b_1$. Then for every $i<j$,
\[
ps(a_i-b_0)\ =\ ps(a_i-b_1)\ =\ p\psi(a_i-a_j).
\]
\end{lemma}
\begin{proof}
Suppose $(a_i)$ is spread out by $b$, and $i<j$ are arbitrary. Then by definition $a_i-b\ll a_j-b$, and by the Successor Identity, 
\[
\psi(a_i-a_j)\ =\ \psi\big((a_j-b)-(a_i-b)\big)\ =\ s(a_i-b). \qedhere
\]
\end{proof}

\noindent
If $(a_i)$ is spread out by $b$, then the sign of the difference $(a_i-b)-ps(a_i-b)$ can depend on $i$ and $b$, although in a very \emph{dependent} way (see Figure~\ref{seqmonotonefigure}):

\begin{lemma}
\label{seqmonotone}
Suppose $(a_i)_{i\in I}$ is spread out by $b$. Let
\[I^*:=
\begin{cases}
I & \text{if $I$ does not have a greatest index} \\
I^{<d} & \text{if $I$ has a greatest index $d$.}
\end{cases}
\]
Then the function $i\mapsto a_i-ps(a_i-b):I^*\to\Gamma$ is either constant, strictly increasing, or strictly decreasing.
\end{lemma}
\begin{proof}
Fix indices $i_0<j_0$ from $I^*$ and fix $\star\in\{=,<,>\}$ such that 
\[
\big(a_{i_0}-ps(a_{i_0}-b)\big)\ \star\ \big(a_{j_0}-ps(a_{j_0}-b)\big).
\]
Next, let $j<k$ be arbitrary indices from $I^*$, and fix an index $d\in I$ which is greater than $j_0$ and $k$. Note that the sequence $(a_i)_{i\in I^{<d}}$ is $a_d$-indiscernible. Thus
\begin{align*}
\big(a_{i_0} - ps(a_{i_0}-b)\big)\ \star\ \big(a_{j_0} - ps(a_{j_0}-b)\big) \ &\Longleftrightarrow\  \big(a_{i_0} -p\psi(a_{i_0}-a_d)\big)\ \star\ \big(a_{j_0}-p\psi(a_{j_0}-a_d)\big) \\
&\Longleftrightarrow \ \big(a_{j} -p\psi(a_{j}-a_d)\big)\ \star\ \big(a_{k}-p\psi(a_{k}-a_d)\big) \\
&\Longleftrightarrow \ \big(a_{j} - ps(a_{j}-b)\big)\ \star\ \big(a_{k} - ps(a_{k}-b)\big).
\end{align*}
Thus the function $i\mapsto a_i-ps(a_i-b):I^*\to\Gamma$ is either constant, strictly increasing, or strictly decreasing, depending on $\star$.
\end{proof}

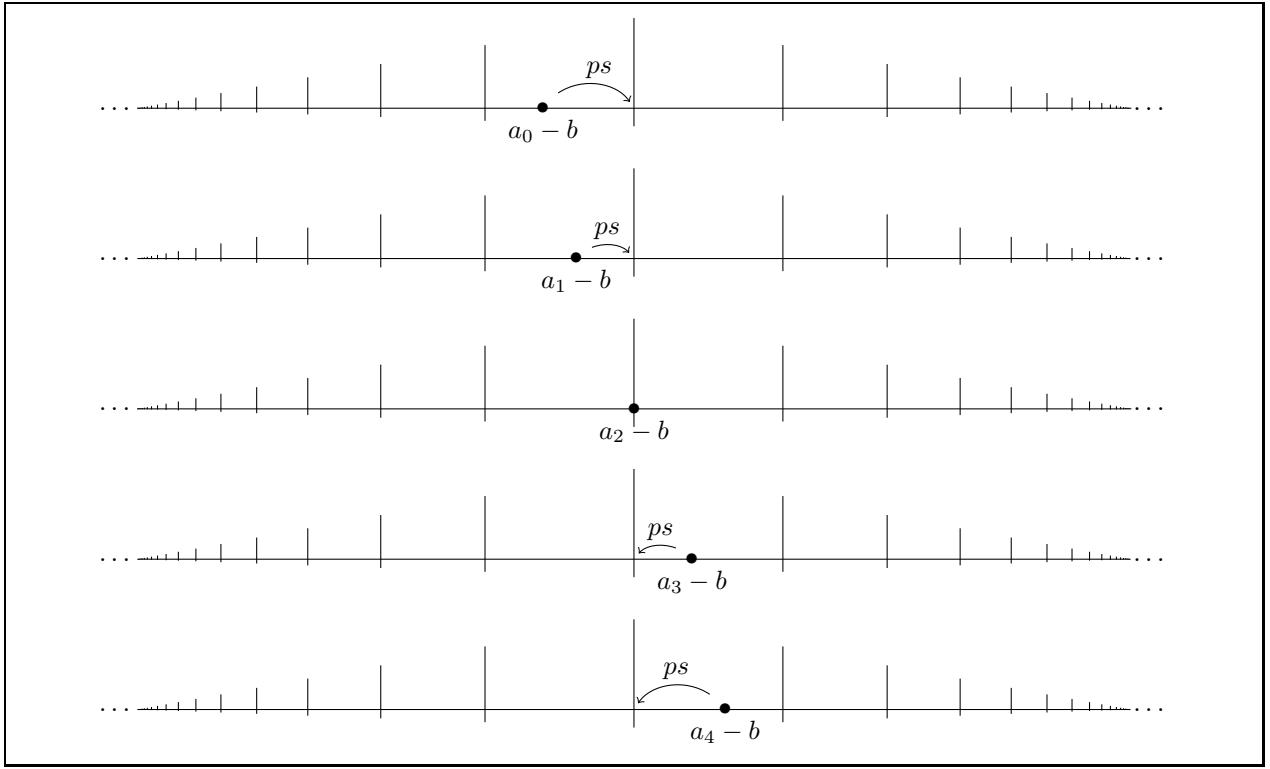
\begin{figure}[h!]
\caption{To illustrate Lemma~\ref{seqmonotone} we apply $ps$ to the sequence $(a_n-b)_{n<5}$ from Figure~\ref{spreadoutfigure}}
\label{seqmonotonefigure}
\begin{center}
\begin{tikzpicture}[x=1.1cm,y=1cm]
\tikzmath{\x = .7; \y = .2; \z =6;}

\foreach \b in {0,1,2,3,4} {
\begin{scope}

\draw (1.5,-2*\b)--(13.5,-2*\b);
\node at (1.25,-2*\b) {$\dots$};
\node at (13.75,-2*\b) {$\dots$};
\draw (7.5, -2*\b-.2*\z*\y)--(7.5, -2*\b+\z*\y);

\foreach \a in {1,2,3,4,5,6,7,8,9,10,11,12,13,14,15,16,17,18,19,20,21,22,23,24,25,26,27,28,29,30} {
\begin{scope}
\draw (7.5+\z-\z*\x^\a, -2*\b-.2*\z*\y*\x^\a)--(7.5+\z-\z*\x^\a, -2*\b+\z*\y*\x^\a);
\draw (7.5-\z+\z*\x^\a, -2*\b-.2*\z*\y*\x^\a)--(7.5-\z+\z*\x^\a, -2*\b+\z*\y*\x^\a);
\end{scope}}

\end{scope}}

\node (A0) at (6.4, 0) {$\bullet$}; \node at ([yshift=-.3cm]A0) {$a_0-b$};
\draw[<-, bend right=45, shorten <=.1cm] 
(7.5,0) to node[yshift= .2cm] {$ps$} (A0);

\node (A1) at (6.8, -2) {$\bullet$}; \node at ([yshift=-.3cm]A1) {$a_1-b$};
\draw[<-, bend right=35, shorten <=.1cm] 
(7.5,-2) to node[yshift= .2cm, xshift = -.1cm] {$ps$} (A1);

\node (A2) at (7.5, -4) {$\bullet$}; \node at ([yshift=-.3cm]A2) {$a_2-b$};

\node (A3) at (8.2, -6) {$\bullet$}; \node at ([yshift=-.3cm]A3) {$a_3-b$};
\draw[->, bend right=35, shorten >=.1cm] 
(A3) to node[yshift= .2cm, xshift = .1cm] {$ps$} (7.5,-6);

\node (A4) at (8.6, -8) {$\bullet$}; \node at ([yshift=-.3cm]A4) {$a_4-b$};
\draw[->, bend right=45, shorten >=.1cm] 
(A4) to node[yshift= .2cm,  xshift = .1cm] {$ps$} (7.5,-8);
\end{tikzpicture}
\end{center}
\end{figure}

\noindent
The following is the key proposition:

\begin{prop}
\label{pnonconstantinfty}
Suppose $I = I_1+(c)+I_2$ is in distal configuration at $c$. Further suppose that $b\in \Gamma$ is such that $(a_i)_{i\in I_1+I_2}$ is $b$-indiscernible. If $p(a_i-b) = \infty$ for every $i\in I_1+I_2$, then $p(a_c-b) = \infty$.
\end{prop}
\begin{proof}
We have a few cases to consider.

\textbf{Case 1:} \emph{There is $i_0\neq c$ such that $a_{i_0}-b>\Psi$.} Then by $b$-indiscernibility and monotonicity, $a_i-b>\Psi$ for every $i\in I$, and so $a_c-b\not\in\Psi$, so $p(a_c-b) = \infty$.

\textbf{Case 2:} \emph{There is $n$ and $i_0\neq c$ such that $a_{i_0}-b\leq s^n0$.} Again by $b$-indiscernibility and monotonicity, we have that $a_i-b<s^n0$ for every $i$. Assume towards a contradiction that $p(a_c-b)\neq\infty$. Then $a_c-b = s^m0$ for some $1<m<n$, and so for $i\neq c$,  $a_i-b<s^m0$ iff $i<c$, contradicting the indiscernibility of $(a_i-b)_{i\in I_1+I_2}$.

We may now assume by monotonicity that $(a_i-b)_{i\in I}\subseteq \Gamma^{>(s^n0)_n}\cap \Psi^{\downarrow}$. In particular, $s0\ll a_i-b$ for every $i$.

\textbf{Case 3:} \emph{The sequence $\big(ps(a_i-b)\big)_{i\in I_1+I_2}$ takes the constant value $b'$.} Then by $b$-indiscernibility and monotonicity, either for all $i$, $a_i-b<b'$ or for all $i$, $a_i-b>b'$. In either case, $a_c-b\neq b'$, however $ps(a_c-b) = b'$. Thus $a_c-b\neq \Psi$.

\textbf{Case 4:} \emph{The sequence $\big(ps(a_i-b)\big)_{i\in I_1+I_2}$ is strictly increasing.} Then the sequence $\big(ps(a_i-b)\big)_{i\in I}$ is strictly increasing, and it follows from $b$-indiscernibility of $(a_i)_{i\in I_1+I_2}$ and $I$ being in distal configuration at $c$ that $(a_i)_{i\in I}$ is spread out by $b$. By Lemma~\ref{seqmonotone} and $b$-indiscernibility of $(a_i)_{i\in I_1+I_2}$, the function $i\mapsto a_i-ps(a_i-b):I\to\Gamma$ is either constant, strictly increasing, or strictly decreasing. By $b$-indiscernibility and the assumption that $p(a_i-b) = \infty$ for every $i\in I_1+I_2$, it follows that this function takes values entirely below $b$, or entirely above $b$. Thus $a_c-ps(a_c-b)\neq b$, and so $a_c-b\not\in \Psi$.
\end{proof}

\section{Extensions}
\label{extensions}

\noindent
\emph{In this section, $(\Gamma,\psi)$ is a model of $T_{\log}$.} Here we will prove the relevant facts that will allow us to later verify condition (2) in Distal Criteria~\ref{distal_multi}, with $(\Gamma,\psi)$ playing the role of ``$(\bm{N},\frak{F})$''. In~\ref{distal_multi}(2), we are allowed to assume the substructure $\bm{M}$ of $\bm{N}$ is closed under every function from $\frak{F}$.

\medskip\noindent
For the case $\frak{f} = \psi$, we do not need the subgroup $\Gamma_0$ of $\Gamma$ to be closed under any of the functions $\psi$, $s$, or $p$:

\begin{prop}
\label{finite_psi_ext}
Suppose $\Gamma_0$ is a divisible ordered subgroup of $\Gamma$. Given $c_1,\ldots,c_m\in\Gamma\setminus\Gamma_0$, we have
\[
\textstyle \# \big(\psi((\Gamma_0+\sum_{i=1}^m\Q c_i)^{\neq})\setminus \psi(\Gamma_0^{\neq})\big)\ \leq\ m.
\]
In particular,
 there is $n\leq m$ and distinct 
\[
\textstyle
d_1,\ldots,d_n\ \in\ \psi\big((\Gamma_0+\sum_{i=1}^m\Q c_i)^{\neq}\big)\setminus \psi(\Gamma_0^{\neq})
\]
 such that
\[
\textstyle
\psi\big((\Gamma_0+\sum_{i=1}^m\Q c_i)^{\neq}\big)\ \subseteq\ \big(\!\bigoplus_{\alpha\in \psi(\Gamma_0^{\neq})}\Q\alpha\big)\! \oplus\! \big(\!\bigoplus_{j=1}^n\Q d_j\big).
\]
\end{prop}
\begin{proof}
This follows by induction on $m$. To simplify notation, we will show the inductive step only for $m=1$. Let $c\in\Gamma\setminus\Gamma_0$. If $\psi(\Gamma_0+\Q c) = \psi(\Gamma_0)$, then we are done. Otherwise, suppose that $\psi(\Gamma_0+\Q c) \neq \psi(\Gamma_0)$. As $\psi$ is constant on archimedean classes, we must have that $[\Gamma_0+\Q c] \neq [\Gamma_0]$. By \cite[Lemma 2.4.4]{ADAMTT}, there is $c^* \in \Gamma_0+\Q c$ with $\left[\Gamma_0+\Q c\right] = \left[\Gamma_0\right] \cup \big\{[c^*]\big\}$ and so
$
\psi\left(\Gamma_0+\Q c\right) = \psi(\Gamma_0) \cup \big\{\psi(c^*)\big\}.
$
\end{proof}

\noindent
For the case $\frak{f} = s$, we will need the subgroup to be closed under the functions $\psi$ and $s$, as the following example illustrates:

\begin{example}
Suppose $(\Gamma,\psi)$ has an element $\alpha\in\Psi$ such that $\alpha>s^n0$ for every $n$. Fix such an element $\alpha$. Let $\Gamma_0$ be the divisible ordered subgroup of $\Gamma$ generated by 
\[
\{s^n0:n\geq 1\}\cup\{\alpha_1-\alpha_0: \alpha_1,\alpha_0\in\Psi\ \&\ s0 \ll \alpha_0<\alpha_1\}.
\]
By Fact~\ref{cor6.5gehretQE} we have
$s(\Gamma_0) = \{s^n0:n\geq 1\}\subseteq\Gamma_0$.
However $\Psi\subseteq \Gamma_0+\Q\alpha$ and thus
\[
\#\big(s(\Gamma_0+\Q\alpha)\setminus s(\Gamma_0)\big)\ =\ \#(\Psi\setminus \{s^n0:n\geq 1\}) \ = \ \infty.
\]
\end{example}

\begin{prop}
\label{finite_s_ext}
Suppose $\Gamma_0$ is a divisible ordered subgroup of $\Gamma$ such that $s(\Gamma_0)\subseteq\Gamma_0$ and $\psi(\Gamma_0^{\neq})\subseteq \Gamma_0$. Given $c_1,\ldots,c_m\in\Gamma\setminus\Gamma_0$, we have
\[
\textstyle \#\big(s(\Gamma_0+\sum_{i=1}^m\Q c_i)\setminus s(\Gamma_0)\big) \ \leq \ m+1.
\]
In particular,
 there is $n\leq m+1$ and distinct 
\[
\textstyle
d_1,\ldots,d_n\ \in\ s\big(\Gamma_0+\sum_{i=1}^m\Q c_i\big)\setminus s(\Gamma_0)
\]
 such that
\[
\textstyle
s\big(\Gamma_0+\sum_{i=1}^m\Q c_i\big)\ \subseteq\ \big(\!\bigoplus_{\alpha\in s(\Gamma_0)}\Q\alpha\big)\! \oplus\! \big(\!\bigoplus_{j=1}^n\Q d_j\big).
\]
\end{prop}
\begin{proof}
Suppose $c_1,\ldots,c_m\in\Gamma\setminus\Gamma_0$ and set $\Gamma_1:=\Gamma_0+\sum_{i=1}^m\Q c_i$. 
Assume towards contradiction that there are $m+2$ distinct elements in $s(\Gamma_1)\setminus s(\Gamma_0)$. Choose $e_1,\ldots,e_{m+2} \in \Gamma_1$ such that $s(e_1)<\ldots<s(e_{m+2})$ and such that $s(e_i) \not\in s(\Gamma_0)$ each  $i=1,\ldots,m+2$. By the Successor Identity, we have that $\psi(e_{i+1}-e_i) = s(e_i)$ for each $i = 1,\ldots,m+1$. The closure assumptions on $\Gamma_0$ imply that $(\Gamma_0,\psi)$ is an asymptotic couple with asymptotic integration, and so $s(\Gamma_0) = \psi(\Gamma_0)$. Thus there are $m+1$ distinct elements in $\psi(\Gamma_1) \setminus \psi(\Gamma_0)$, contradicting Proposition \ref{finite_psi_ext} above. \qedhere
\end{proof}

\begin{prop}
\label{finite_p_ext}
Suppose $(\Gamma_0,\psi)\preccurlyeq(\Gamma,\psi)$. Given $c_1,\ldots,c_m\in\Gamma\setminus\Gamma_0$, we have
\[
\textstyle \# \big(p(\Gamma_0+\sum_{i=1}^m\Q c_i)\setminus p(\Gamma_0)\big)\ \leq \ m+1
\]
In particular, there is $n\leq m+1$ and distinct 
\[
\textstyle d_1,\ldots,d_n \ \in\ p\big(\Gamma_0+\sum_{i=1}^m\Q c_i\big)\setminus p(\Gamma_0)
\]
such that
\[
\textstyle p\big(\Gamma_0+\sum_{i=1}^m\Q c_i\big)\ \subseteq\ \big(\!\bigoplus_{\alpha\in p(\Gamma_0), \alpha\neq\infty}\Q\alpha\big)\!\oplus\!\big(\!\bigoplus_{j=1}^n\Q d_j\big)\cup\{\infty\}.
\]
\end{prop}
\begin{proof}
Suppose $c_1,\ldots,c_m\in\Gamma\setminus\Gamma_0$ and set $\Gamma_1:=\Gamma_0+\sum_{i=1}^m\Q c_i$. Assume towards contradiction there are $m+2$ elements $e_1,\ldots,e_{m+2} \in \Gamma_1$ such that $p(e_i) \in p(\Gamma_1)\setminus p(\Gamma_0)$ for each $i$ and such that $p(e_i) \neq p(e_j)$ for all $i \neq j$. Then $e_i$ is in $\Psi$ for each $i$ and, as $s$ is injective on $\Psi$, we have that $s(e_i) \neq s(e_j)$ for all $i \neq j$. As $p(\Gamma_0) = s(\Gamma_0) \cup \{\infty\}= (\Psi \cap \Gamma_0)\cup \{\infty\}$, we have that $s(e_i) \not\in s(\Gamma_0)$ each  $i$, contradicting Proposition \ref{finite_s_ext}.
\end{proof}

\section{Proof of Theorem~\ref{tlogdistalthm}}
\label{distalproof}

\noindent
In this section we prove Theorem~\ref{tlogdistalthm} by applying Distal Criterion~\ref{distal_multi}. In the language of~\ref{distal_multi}, the role of $T$ will be played by the reduct $T:=T_{\log}\!\upharpoonright\!\L$, where $\L = \{0,+,-,<,(\delta_n)_{n<\omega},\infty\}$. The $\L$-theory $T$ is essentially the same as the theory of ordered divisible abelian groups, except that it contains the element $\infty$ which serves as a default value ``at infinity''. It follows that $T$ is o-minimal and therefore it is distal by~\cite[Lemma 2.10]{SimonDistal}, since o-minimal theories are $\DP$-minimal.

\medskip\noindent
 We now construe $T_{\log}$ as $T_{\log} = T(\frak{F})$, with $\frak{F} = \{\psi,s,p\}$. In particular, $\L_{\log} = \L(\frak{F})$. By~\cite[Theorem 5.2]{gehretQE}, $T(\frak{F})$ has quantifier elimination, which is condition (1) in~\ref{distal_multi}. Verifying condition (2) in~\ref{distal_multi} involves three cases: $\frak{f} = \psi$, $s$, and $p$. These cases are handled respectively by Propositions~\ref{finite_psi_ext},~\ref{finite_s_ext}, and~\ref{finite_p_ext}. Concerning Proposition~\ref{finite_p_ext}, recall that $T_{\log}$ has quantifier elimination and a universal axiomatization (Theorem~\ref{Tlogknownthms}). Thus, if $(\Gamma,\psi)\models T_{\log}$, and $\Gamma_0\subseteq\Gamma$ is a divisible subgroup closed under $s,\psi,$ and $p$, then $\Gamma_0$ is the underlying set of an elementary substructure of $(\Gamma,\psi)$.
 
 \medskip\noindent
 Finally, we will show how to verify condition (3) in~\ref{distal_multi}. Fix a monster model $\M$ of $T$ with underlying set $\Gamma_{\infty}$. Let $\frak{f}\in \frak{F}$, and let $g,h$ be $\L$-terms of arities $n+k$ and $m+l$ respectively, with $m\leq n$, $b_1\in\M^k$, $b_2\in\frak{f}(\M)^l$, $(a_i)_{i\in I}$ be an indiscernible sequence from $\frak{f}(\M)^m\times{\mathbb{M}}^{n-m}$ such that
 \begin{enumerate}[(a)]
 \item $I=I_1+(c)+I_2$ is in distal configuration at $c$, and $(a_i)_{i\in I_1+I_2}$ is $b_1b_2$-indiscernible, and
 \item $\frak{f}\big(g(a_i,b_1)\big) = h(a_i,b_2)$ for every $i\in I_1+I_2$.
 \end{enumerate}
 Our job is to show that $\frak{f}\big(g(a_c,b_1)\big) = h(a_c,b_2)$.
 We have several cases to consider:
 
 \textbf{Case 1:} \emph{$\big(g(a_i,b_1)\big)_{i\in I_1+I_2}$ is a constant sequence.} In this case, $\frak{f}\big(g(a_c,b_1)\big) = h(a_c,b_2)$ follows from Lemma~\ref{constantconstant}.
 
For the remainder of the proof, we assume that \emph{$\big(g(a_i,b_1)\big)_{i\in I_1+I_2}$ is not a constant sequence.} In particular, the symbol $\infty$ does not play a non-dummy role in $g(a_i,b_1)$, so the $\L$-term $g(x,y)$ is essentially a $\Q$-linear combination of its arguments. By grouping these $\Q$-linear combinations, we get $b\in\Gamma$, and a nonconstant indiscernible sequence $(a_i')_{i\in I}$ from $\M$ such that 
\begin{enumerate}[(a)]
\setcounter{enumi}{2}
\item $g(a_i,b_1) = a_i'-b$ for every $i\in I$,
\item $(a_ia_i')_{i\in I}$ is an indiscernible sequence from $\frak{f}(\M)^m\times \M^{n-m+1}$,
\item $(a_ia_i')_{i\in I_1+I_2}$ is $b_1b_2b$-indiscernible, and
\item $\frak{f}(a_i'-b) = h(a_i,b_2)$ for every $i\in I_1+I_2$.
\end{enumerate}
Our job now is to show that $\frak{f}(a_c'-b) = h(a_c,b_2)$. Since $\frak{f}(\M)\subseteq \Psi_{\infty}$ for each $\frak{f}$, by Lemma~\ref{RHSsimplification} we get three more cases:

\textbf{Case 2:} \emph{$h(a_i,b_2)=\infty$ for every $i\in I$.} By Lemma~\ref{nonconstantconstant}(3), this case cannot happen for $\frak{f}\in\{\psi,s\}$. If $\frak{f} = p$, then this case is handled by Proposition~\ref{pnonconstantinfty}.

\textbf{Case 3:} \emph{There is $\beta\in\Psi$ such that $h(a_i,b_2) = \beta$ for every $i\in I$.} If $\frak{f} = \psi$ or $\frak{f} = s$, then this case is handled by Lemma~\ref{nonconstantconstant}(1) or (2). By Lemma~\ref{nonconstantconstant}(4), this case cannot happen for $\frak{f} = p$.

\textbf{Case 4:} \emph{There is $l\in \{1,\ldots,m\}$ such that $h(a_i,b_2) = a_{i,l}$ for every $i\in I$.} This case is handled by Lemma~\ref{nonconstantnonconstant}.

This completes the verification of condition (3) in~\ref{distal_multi} and so we are done with our proof of Theorem~\ref{tlogdistalthm}.

\section{DP-rank}
\label{dpranksection}

\noindent
In this section, we weigh in on the $\DP$-rank of $T_{\log}$. In contrast to distality, a notion of \emph{pureness}, $\DP$-rank gives rise to a certain measure of \emph{diversity} (in the sense that $\DP$-rank measures the diversity of realizations of types as viewed by external parameters; see the Introduction to~\cite{dprkadditivity}). Below we show that $T_{\log}$ is not strongly dependent, and so its $\DP$-rank is quite large. This also underscores the point of view that distality is \emph{not} to be taken as a notion of tameness.
For a concise definition of $\DP$-rank, $\DP$-minimality, and strongly dependent theories, see~\cite{Usvyatsov}. 
See also~\cite[Chapter 4]{SimonNIP} for more information.
We will not define these concepts here and will instead use Proposition~\ref{goodrickprop} as a black box for establishing our negative results.

\begin{thm}
\label{TACnotstrong}
$T_{\log}$ is not strongly dependent. Therefore it is not $\DP$-minimal and does not have finite $\DP$-rank.
\end{thm}

\noindent
It is sufficient to show that $T_{\log}$ is not \emph{strong}, since if a theory is strongly dependent, then it is strong (see~\cite{Goodrick}). To do this, we will use the following criterion:

\begin{prop}\cite[2.14]{Goodrick}
\label{goodrickprop}
Suppose that $\bm{M} = (M;+,<,\ldots)$ is an expansion of a densely-ordered abelian group. Let $\bm{N}$ be a saturated model of $\Th(\bm{M})$, and suppose that for every $\epsilon>0$ in $\bm{N}$ there is an infinite definable discrete set $X\subseteq\bm{N}$ such that $X\subseteq (0,\epsilon)$. Then $\Th(\bm{M})$ is not strong.
\end{prop}

\begin{proof}[Proof of Theorem~\ref{TACnotstrong}]
Let $\bm{N}$ be a saturated model of $T_{\log}$. The infinite definable set $\Psi_{\bm{N}}$ is discrete and has the property that for every $\alpha\in\Psi_{\bm{N}}$, the set $\Psi_{\bm{N}}^{>\alpha}$ is also infinite and discrete. Let $\epsilon>0$ and take $\alpha\in\Psi_{\bm{N}}$ such that $\big(\alpha + 2(s\alpha-\alpha)\big) - \alpha = -2\int\alpha<\epsilon$. Note that then $\alpha+2(s\alpha-\alpha)>\Psi_{\bm{N}}$ by Fact~\ref{sfacts}(\ref{poptop}). The definable infinite discrete set $X:= \Psi_{\bm{N}}^{>\alpha}-\alpha$ has the desired property.
\end{proof}

\section*{Acknowledgements}
\noindent
The authors thank Matthias Aschenbrenner, Artem Chernikov, Lou van den Dries, John Goodrick, Philipp Hieronymi, and Travis Nell for various conversations and correspondences around the topics in this paper.
The first author is supported by the National Science Foundation under Award No. 1703709.

\bibliographystyle{amsplain}	
\bibliography{refs}

\providecommand{\bysame}{\leavevmode\hbox to3em{\hrulefill}\thinspace}
\providecommand{\MR}{\relax\ifhmode\unskip\space\fi MR }
\providecommand{\MRhref}[2]{%
  \href{http://www.ams.org/mathscinet-getitem?mr=#1}{#2}
}
\providecommand{\href}[2]{#2}
\begin{thebibliography}{10}

\bibitem{ADAMTT}
Matthias Aschenbrenner, Lou van~den Dries, and Joris van~der Hoeven,
  \emph{Asymptotic differential algebra and model theory of transseries},
  Annals of Mathematics Studies, vol. 195, Princeton University Press,
  Princeton, NJ, 2017. \MR{3585498}

\bibitem{ChernikovGalvinStarchenko}
Artem Chernikov, David Galvin, and Sergei Starchenko, \emph{Cutting lemma and
  zarankiewicz's problem in distal structures}, 2016.

\bibitem{ChernikovStarchenko}
Artem Chernikov and Sergei Starchenko, \emph{Regularity lemma for distal
  structures}, Journal of the European Mathematical Society.

\bibitem{Goodrick}
Alfred Dolich and John Goodrick, \emph{Strong theories of ordered {A}belian
  groups}, Fund. Math. \textbf{236} (2017), no.~3, 269--296. \MR{3600762}

\bibitem{gehretQE}
Allen Gehret, \emph{The asymptotic couple of the field of logarithmic
  transseries}, J. Algebra \textbf{470} (2017), 1--36. \MR{3565423}

\bibitem{GehretNIP}
\bysame, \emph{N{IP} for the asymptotic couple of the field of logarithmic
  transseries}, J. Symb. Log. \textbf{82} (2017), no.~1, 35--61. \MR{3631276}

\bibitem{GehretLiouville}
\bysame, \emph{A tale of two {L}iouville closures}, Pacific J. Math.
  \textbf{290} (2017), no.~1, 41--76. \MR{3673079}

\bibitem{dependentpairs}
Ayhan G\"unaydin and Philipp Hieronymi, \emph{Dependent pairs}, J. Symbolic
  Logic \textbf{76} (2011), no.~2, 377--390. \MR{2830406}

\bibitem{Hieronymi_Nell}
Philipp Hieronymi and Travis Nell, \emph{Distal and non-distal pairs}, J. Symb.
  Log. \textbf{82} (2017), no.~1, 375--383. \MR{3631293}

\bibitem{dprkadditivity}
Itay Kaplan, Alf Onshuus, and Alexander Usvyatsov, \emph{Additivity of the
  dp-rank}, Trans. Amer. Math. Soc. \textbf{365} (2013), no.~11, 5783--5804.
  \MR{3091265}

\bibitem{rosenlicht}
Maxwell Rosenlicht, \emph{On the value group of a differential valuation.
  {II}}, Amer. J. Math. \textbf{103} (1981), no.~5, 977--996. \MR{630775}

\bibitem{SimonDistal}
Pierre Simon, \emph{Distal and non-distal {NIP} theories}, Ann. Pure Appl.
  Logic \textbf{164} (2013), no.~3, 294--318. \MR{3001548}

\bibitem{SimonNIP}
\bysame, \emph{A guide to {NIP} theories}, Lecture Notes in Logic, vol.~44,
  Association for Symbolic Logic, Chicago, IL; Cambridge Scientific Publishers,
  Cambridge, 2015. \MR{3560428}

\bibitem{TentZiegler}
Katrin Tent and Martin Ziegler, \emph{A course in model theory}, Lecture Notes
  in Logic, vol.~40, Association for Symbolic Logic, La Jolla, CA; Cambridge
  University Press, Cambridge, 2012. \MR{2908005}

\bibitem{Usvyatsov}
Alexander Usvyatsov, \emph{On generically stable types in dependent theories},
  J. Symbolic Logic \textbf{74} (2009), no.~1, 216--250. \MR{2499428}

\end{thebibliography}

\end{document}